\documentclass{amsart}

\usepackage{amsfonts}
\usepackage{amssymb}
\usepackage{amsthm}
\usepackage{eucal}
\usepackage{tikz}

\theoremstyle{plain}
\newtheorem{theorem}{Theorem}[section]

\newtheorem{lemma}[theorem]{Lemma}
\newtheorem{corollary}[theorem]{Corollary}
\newtheorem{conjecture}[theorem]{Conjecture}

\theoremstyle{definition}
\newtheorem*{definition}{Definition}

\newtheorem*{example}{Example}

\theoremstyle{remark}

\newcommand{\diam}{\mathrm{diam}}
\newcommand{\p}{\mathcal{P}}
\newcommand{\gen}[1]{\langle#1\rangle}
\newcommand{\calH}{\mathcal{H}}
\newcommand{\calK}{\mathcal{K}}
\newcommand{\calL}{\mathcal{L}}
\newcommand{\calU}{\mathcal{U}}
\newcommand{\Aut}{\mathrm{Aut}}
\newcommand{\Z}{\mathbb{Z}}

\newcommand{\Supp}{\mathrm{Supp}}
\newcommand{\normal}{\trianglelefteq}

\begin{document}
\title[Proper power graphs]{On the connectivity of proper power graphs of finite groups}
\author{A. Doostabadi}
\author{M. Farrokhi D. G.}
\address{Department of Pure Mathematics, Ferdowsi University of Mashhad, Mashhad, Iran}
\email{a.doostabadi@yahoo.com}
\address{Department of Pure Mathematics, Ferdowsi University of Mashhad, Mashhad, Iran}
\email{m.farrokhi.d.g@gmail.com}
\date{}
\begin{abstract}
We study the connectivity of proper power graphs of some family of finite groups including nilpotent groups, groups with a non-trivial partition, and symmetric and alternating groups.
\end{abstract}
\maketitle
\section{Introduction}
The power graph $\p(G)$ of a group $G$, is the graph whose vertex set is the group $G$ such that two distinct elements are adjacent if one is a power of the other. Kelarev and Quinn \cite{avk-sjq:2002} were the first who studied the directed power graph of semigroups, in which there is an arc from a vertex $x$ to a vertex $y$ if $y$ is positive power of $x$. Chakrabarty, Ghosh and Sen \cite{ic-sg-mks:2009} considered the power graphs of the multiplicative semigroup $\Z_n$ and its group of units $U_n$, and determined when such graphs are complete, planar or Hamiltonian. Cameron and Ghosh \cite{pjc-sg:2011} proved that two finite Abelian groups are isomorphic if and only if they have isomorphic power graphs. Also, Cameron \cite{pjc:2010} shows that two groups with isomorphic power graphs should have the same number of elements of each order. However, the converse to this theorem does not hold in general, for non-isomorphic $p$-groups of the same order of exponent $p$, have isomorphic power graphs.

Doostabadi, Erfanian and Jafarzadeh \cite{ad-ae-aj} discussed some graph theoretical properties of power graphs including planarity and perfectness, and obtained the independence and chromatic numbers of power graphs of cyclic and arbitrary groups, respectively. Clearly, the power graph $\p(G)$ of a finite group $G$ is always connected since all the vertices are adjacent to the identity element. However, the proper power graph $\p^*(G)$ obtained by removing the identity element from the power graph $\p(G)$ of a given finite group $G$, does not remain connected in general. The aim of this paper is to answer partially this question; whether the proper power graph of a finite group is connected? We shall determine all connected components of the proper power graphs of some classes of groups involving nilpotent groups, groups with a non-trivial partition, and symmetric and alternating groups. All over this paper, the graphs under investigations are proper power graphs, and it will be used without further reference. Given a finite group $G$, and nontrivial elements $x,y\in G$, the notation $x\sim y$ indicates that $x$ and $y$ as vertices of the proper power graph are adjacent. Also, the notation $x\simeq y$ indicates that either $x=y$ or $x\sim y$. A path is an alternative sequence of vertices and edges which begins and ends at a pair of vertices such that consecutive terms are incident. The length of a path is the number of edges in the path. The distance $d(x,y)$ between two vertices $x$ and $y$ is the length of the shortest path starting and ending at $x$ and $y$, respectively. Also, the diameter of a connected graph $\Gamma$, denoted by $\diam(\Gamma)$, is the maximum distance between all pairs of vertices of $\Gamma$. If $\Gamma$ is disconnected, then we set $\diam(\Gamma)=\infty$. The number of all connected components of a graph $\Gamma$ is denoted by $c(\Gamma)$.
\section{Nilpotent groups}
In this section, we shall investigate the connectivity of the proper power graph of finite nilpotent groups. To achieve this goal, first we give some general results for distances between elements of prime orders and the connectivity of proper power graphs when the center of the underlying groups have special structures. The following lemma will be used frequently in the paper.
\begin{lemma}\label{coprimeelements}
Let $G$ be a finite group and $x,y\in G\setminus\{1\}$ such that $xy=yx$ and $\gcd(|x|,|y|)=1$. Then $x\sim xy\sim y$.
\end{lemma}
As a simple and immediate consequence of the above lemma we have
\begin{corollary}\label{uniqueinvolution}
Let $G$ be a finite group which has a unique involution. Then $\p^*(G)$ is connected.
\end{corollary}
\begin{proof}
Suppose $x$ is the unique involution of $G$. Clearly $x\in Z(G)$. If $y\in G$ is a nontrivial element, then either $x\in\gen{y}$, which implies that $y\sim x$, or $|y|$ is odd and by Lemma \ref{coprimeelements}, $y\sim xy\sim x$. Hence $\p^*(G)$ is connected.
\end{proof}
\begin{lemma}\label{primorderelements}
Let $G$ be a finite group and $x,y\in G\setminus\{1\}$ be elements of prime orders. Then
\begin{itemize}
\item[(1)]$d(x,y)=1$ if and only if $x\neq y$ and $\gen{x}=\gen{y}$,
\item[(2)]$d(x,y)=2$ if and only if $xy=yx$ and $|x|\neq|y|$,
\item[(3)]$d(x,y)$ is never an odd number greater than one,
\item[(4)]$d(x,y)=2n$ if and only if there exist $n-1$ elements $g_1,\ldots,g_{n-1}\in G$ of prime orders, but not fewer, such that $|g_i|\neq|g_{i+1}|$ and $g_ig_{i+1}=g_{i+1}g_i$ for $0\leq i<n$, in which $g_0=x$ and $g_n=y$.
\end{itemize}
\end{lemma}
\begin{proof}
The parts (1), (2) and (3) are straightforward. To prove (4), let 
\[x=z_0\sim z_1\sim z_2\sim\cdots\sim z_{2n-1}\sim z_{2n}=y\]
be a path of length $2n$ between $x$ and $y$. Clearly $z_0,z_2\in\gen{z_1}$. Since $|z_0|$ is a prime it follows that $\gcd(|z_0|,|z_2|)=1$ for otherwise we may omit $z_1$ and reach to a shorter path. Replacing $z_2$ by a suitable power, we may assume without loss of generality that $|z_2|\neq|z_0|$ is a prime. Then we may replace $z_1$ by $z_0z_2$ and get a new path of length $2n$ between $x$ and $y$ in which $z_1=z_0z_2=z_2z_0$ and $|z_2|$ is a prime. Continuing this way, we may assume that $|z_{2i}|$ is a prime and $z_{2i-1}=z_{2(i-1)}z_{2i}=z_{2i}z_{2(i-1)}$, for all $i=1,\ldots,n$. Let $g_i=z_{2i}$ for $i=0,\ldots,n$. Then $g_0,\ldots,g_n$ satisfy the given properties and the proof is complete.
\end{proof}
\begin{theorem}\label{nonpgroupcenter}
Let $G$ be a finite group such that $Z(G)$ is not a $p$-group. Then $\p^*(G)$ is connected. Moreover, $\diam(\p^*(G))\leq6$ and the bound is sharp.
\end{theorem}
\begin{proof}
Let $z_p,z_q\in Z(G)$ be elements of order $p$ and $q$, respectively, where $p$ and $q$ are distinct primes. Also, let
$x,y\in G$ be arbitrary elements, $|x|=m$ and $|y|=n$. If $\gcd(m,n)\neq1$ and $r$ is a common prime divisor of $m$ and $n$, then for $s\in\{p,q\}\setminus\{r\}$ we have $x\sim x^{\frac{m}{r}}\sim x^{\frac{m}{r}}z_s\sim z_s\sim y^{\frac{n}{r}}z_s\sim y^{\frac{n}{r}}\sim y$ and $d(x,y)\leq6$. If $\gcd(m,n)=1$, then we have three cases. If $r\in\{p,q\}$ and $r\nmid m,n$, then $x\sim xz_r\sim z_r\sim yz_r\sim y$ and $d(x,y)\leq4$. If $\{r,s\}=\{p,q\}$ and $r|m$ and $s\nmid n$, then $x\sim x^{\frac{m}{r}}\sim x^{\frac{m}{r}}z_s\sim z_s\sim yz_s\sim y$ and $d(x,y)\leq5$. Also, if $\{r,s\}=\{p,q\}$ and $r|m$ and $s|n$, then $x\sim xz_s\sim z_s\sim z_rz_s\sim z_r\sim yz_r\sim y$ and $d(x,y)\leq6$. Therefore, $\diam(\p^*(G))\leq6$. Now, let $G=\Z_6\times S_3$, where $S_3=\gen{x,y:x^3=y^2=1,x^y=x^{-1}}$ is the symmetric group of degree $3$. Then $Z(G)\cong\Z_6$ is not a $p$-group and by Lemma \ref{primorderelements}, $d(x,y)=6$. Hence the bound is sharp.
\end{proof}

The connectivity of proper power graph of finite nilpotent groups depends on the structure of their Sylow $p$-subgroups, hence we first consider the case of finite $p$-groups.
\begin{theorem}\label{pgroupcomponents}
Let $G$ be a finite $p$-group. Then there exists a one-to-one correspondence between the connected components of $\p^*(G)$ and the minimal cyclic subgroups of $G$. 
\end{theorem}
\begin{proof}
Let $\mathcal{M}$ be the set of all minimal cyclic subgroups of $G$ and $\mathcal{C}$ be the set of all connected components of $\p^*(G)$. Let $f$ be the map defined as follows:
\[\begin{array}{rcl}
f:\mathcal{M}&\longrightarrow&\mathcal{C}\\
\gen{x}&\longrightarrow&C_{\gen{x}},
\end{array}\]
where $C_{\gen{x}}$ is the connected component of $\p^*(G)$ containing $x$. We show that $f$ is a well-defined bijection. Since an element of order $p$ is adjacent to its non-trivial power, the map $f$ is well-defined. If $x$ and $y$ are adjacent vertices of $\p^*(G)$, then $\Omega_1(\gen{x})=\Omega_1(\gen{y})$, which implies that the subgroups generated by single elements in a connected component of $\p^*(G)$, share the same subgroup of order $p$. On the other hand, if $x$ and $y$ are elements of $G$, for which $\gen{x}$ and $\gen{y}$ have the same subgroup $\gen{z}$ of order $p$, then $x\simeq z\simeq y$. Hence $f$ is bijective and the proof is complete.
\end{proof}
\begin{example}
If $G$ is a finite $p$-group of exponent $p$, then $\p^*(G)$ is a union of complete graphs of order $p-1$. Moreover, the number of connected components of $\p^*(G)$ is equal to $(p^n-1)/(p-1)$, where $p^n$ is the order of $G$.
\end{example}
\begin{theorem}\label{nilpotentconnectivity}
Let $G$ be a finite nilpotent group.
\begin{itemize}
\item[(1)]if $G$ is a $p$-group, then the number of connected components of $\p^*(G)$ is the same as the number of subgroups of $G$ of order $p$. In particular, $\p^*(G)$ is connected if and only if $G$ is a cyclic $p$-group or a generalized quaternion $2$-group,
\item[(2)]if $G$ is not a $p$-group and the Sylow $p$-subgroups of $G$ each of which is a cyclic $p$-group or a generalized quaternion $2$-group, then $\p^*(G)$ is connected and $\diam(\p^*(G))=2$.
\item[(3)]if $G$ is not a $p$-group and $G$ has a Sylow $p$-subgroup, which is neither a cyclic $p$-group nor a generalized quaternion $2$-group, then $\p^*(G)$ is connected and $\diam(\p^*(G))=4$.
\end{itemize}
\end{theorem}
\begin{proof}
(1) It is straightforward by Theorem \ref{pgroupcomponents} and \cite[5.3.6]{djsr:1982}.

(2) Let $x,y\in G\setminus\{1\}$ be non-adjacent vertices, $|x|=m$ and $|y|=n$. If $\gcd(m,n)=1$, then by Lemma \ref{coprimeelements}, $d(x,y)=2$. Also, if $\gcd(m,n)\neq1$, and $p$ is a common prime divisor of $m$ and $n$, then there exists an integer $t$ such that $x^{\frac{tm}{p}}=y^{\frac{n}{p}}$. Hence $x\sim y^{\frac{n}{p}}\sim y$ so that $d(x,y)=2$. Therefore $\diam(\p^*(G))=2$.

(3) Let $x,y\in G\setminus\{1\}$, $|x|=m$ and $|y|=n$. If $p,q\in\pi(G)$ are distinct primes, $p|m$ and $q|n$, then $x\sim x^{\frac{m}{p}}\sim x^{\frac{m}{p}}y^{\frac{m}{q}}\sim y^{\frac{m}{q}}\sim y$ and $d(x,y)\leq4$. Otherwise $x$ and $y$ are $p$-elements for some prime $p$. Thus $x\sim xz\sim z\sim yz\sim y$ for each $q$-element $z$, where $q\neq p$ that is $d(x,y)\leq4$. Therefore $\diam(\p^*(G))\leq4$. On the other hand, if $x,y$ are nontrivial $p$-elements such that $\gen{x}\cap\gen{y}=1$, then it is easy to see that $d(x,y)=4$. Therefore $\diam(\p^*(G))=4$.
\end{proof}
\begin{corollary}
Let $G$ be a non-cyclic finite group. If $\p^*(G)$ is connected, then there exists a nontrivial element $x\in G$ and distinct cycles $\gen{y}$ and $\gen{z}$ of the same order such that $x\in\gen{y}\cap\gen{z}$.
\end{corollary}
\begin{proof}
Suppose on the contrary that $\gen{y}=\gen{z}$ for all $x,y,z\in G$ such that $|y|=|z|$ and $x\in\gen{y}\cap\gen{z}$. A simple verification shows that $N_G(\gen{x})\subseteq N_G(\gen{y})$ whenever $1\neq\gen{x}\leq\gen{y}$. On the other hand, $N_G(\gen{y})\subseteq N_G(\gen{x})$, which implies that $N_G(\gen{x})=N_G(\gen{y})$ whenever $1\neq\gen{x}\leq\gen{y}$. Since $\p^*(G)$ is connected, it follows that $N_G(x)=N_G(y)$ for all $x,y\in G\setminus\{1\}$. Therefore $G$ is a Dedeking group and by \cite[5.3.7]{djsr:1982}, $G$ is nilpotent. If $G$ is a $p$-group, then by Theorem \ref{nilpotentconnectivity}(1), $G$ is a cyclic group, which is a contradiction. Thus $G$ is not a $p$-group. If $P$ is a Sylow $p$-subgroup of $G$ which is not cyclic, then there exist elements $x,y\in P$ of the same order such that $\gen{x}\neq\gen{y}$. Now if $z$ is a $q$-element ($q\neq p$) of $G$, then $|xz|=|yz|$, $\gen{xz} \neq\gen{yz}$ and $z\in\gen{xz}\cap\gen{yz}$, which contradicts the assumption. Thus, the Sylow $p$-subgroups of $G$ each of which is cyclic, which implies that $G$ is a cyclic group, a contradiction. The proof is complete.
\end{proof}

Utilizing the above results, all finite groups with proper power graphs of diameter at most three can be classified.
\begin{theorem}
Let $G$ be a finite group. Then
\begin{itemize}
\item[(1)]$\diam(\p^*(G))=1$ if and only if $G$ is a cyclic $p$-group,
\item[(2)]$\diam(\p^*(G))=2$ if and only if $G$ is nilpotent which is not a cyclic $p-$group and the Sylow $p$-subgroups of $G$ each of which is a cyclic $p$-group or a generalized quaternion $2$-group.
\item[(3)]$\diam(\p^*(G))=3$ if and only if $G$ is not nilpotent and $G$ has exactly one subgroup of order $p$ for all $p\in\pi(G)$.
\end{itemize}
\end{theorem}
\begin{proof}
(1) It follows from \cite[Theorem 2.12]{ic-sg-mks:2009}.

(2) If $\diam(\p^*(G))=2$ and $x,y\in G\setminus\{1\}$ such that $\gcd(|x|,|y|)=1$, then $d(x,y)=2$ and there exists $g\in G\setminus\{1\}$ such that $x\sim g\sim y$. Thus $x,y\in\gen{g}$ and hence $x$ and $y$ commute. Therefore $G$ is nilpotent. Now, the result follows by Theorem \ref{nilpotentconnectivity}.

(3) First suppose that $\diam(\p^*(G))=3$. If $x,y\in G\setminus\{1\}$ are elements of prime order $p$, then by Lemma \ref{primorderelements}(1,2,3), $\gen{x}=\gen{y}$. Hence $G$ has exactly one subgroup of order $p$ for all $p\in\pi(G)$. If $G$ is nilpotent, then by parts (1,2), $\diam(\p^*(G))\leq2$, which is a contradiction. Thus $G$ is not nilpotent. Now, suppose the converse is true. Then by parts (1,2), $\diam(\p^*(G))\geq3$. Suppose on the contrary that $\diam(\p^*(G))>3$ and  $x,y\in G\setminus\{1\}$ such that $d(x,y)>3$. Let $|x|=m$ and $|y|=n$. If $\gcd(m,n)\neq1$ and $p$ is a common prime divisor of $m$ and $n$, then $\gen{x^{\frac{m}{p}}}=\gen{y^{\frac{n}{p}}}$ so that $x\simeq x^{\frac{m}{p}}\simeq y$, which is a contradiction. Thus $\gcd(m,n)=1$. Let $p=\min\pi(\gen{x})\cup\pi(\gen{y})$. Without loss of generality, we assume that $p\in\pi(\gen{x})$. Let $H=\gen{x^{\frac{m}{p}},y}$. Then $\gen{x^{\frac{m}{p}}}\normal H$ and $H/C_H(x^{\frac{m}{p}})$ is isomorphic to a subgroup of $\Aut(\gen{x^{\frac{m}{p}}})\cong\Z_{p-1}$. Since $p=\min\pi(H)$, it follows that the automizer of $\gen{x^{\frac{m}{p}}}$ is trivial that is $H=C_H(x^{\frac{m}{p}})$. Hence $x^{\frac{m}{p}}$ and $y$ commutes and consequently $x\simeq x^{\frac{m}{p}}\sim x^{\frac{m}{p}}y\sim y$, which is a contradiction. Therefore $\diam(\p^*(G))=3$.
\end{proof}
\section{Finite groups with a non-trivial partition}
If $G$ is a finite group and $H,K$ are subgroups of $G$ such that $H\cap K=1$, then the subgraph of $\p^*(G)$ induced by $H\cup K\setminus\{1\}$ is the disjoint union of two subgraphs induced by $H\setminus\{1\}$ and $K\setminus\{1\}$. This motivates us to study finite groups that are union of some disjoint subgroups. We begin with giving formal definitions and known results on such groups.

\begin{definition}
Let $G$ be a non-trivial group. The set $\Pi$ of non-identity subgroups of $G$ is a \textit{partition} of $G$ if every non-identity element of $G$ belongs to exactly one subgroup in $\Pi$, in other words
\begin{itemize} 
\item[(1)]The union of subgroups in $\Pi$ is the group $G$,
\item[(2)] If $X,Y \in \Pi$ and $X\neq Y$, then $X\cap Y=1$.
\end{itemize}
Also, a partition $\Pi$ is \textit{non-trivial} if for every $X\in\Pi$, $X\neq G$. The subgroups in $\Pi$ are called the \textit{components} of partition.
\end{definition}
\begin{definition}
Let $G$ be a group and $p$ be a prime. Then the \textit{Hughes subgroup} $H_p(G)$ of $G$ is defined as the subgroup generated by all elements of $G$ whose orders are different from $p$.
\end{definition}
\begin{definition}
A finite group $G$ is called a \textit{Hughes-Thompson group} if $G$ is not a $p$-group and $H_p(G)\subset G$ for some prime divisor $p$ of $|G|$.
\end{definition}
\begin{theorem}[\cite{gz:2003}]\label{partition}
Let $G$ be a finite group with a non-trivial partition. Then $G$ is isomorphism to one of the following groups:
\begin{itemize}
\item[(1)]a non-cyclic $p$-group with $H_p(G)\neq G$,
\item[(2)]a group of Hughes-Thompson type,
\item[(3)]a Frobenius group,
\item[(4)]$PGL(2,p^n)$ ($p$ odd),
\item[(5)]$PSL(2,p^n)$,
\item[(6)]a Suzuki group $Sz(2^{2n+1})$,
\end{itemize}
where $p$ is a prime number.
\end{theorem}

As it is mentioned before, if $G$ is a finite group with a non-trivial partition $\Pi$, then each connected component of $\p^*(G)$ is a connected component of the subgraph induced by some component of $\Pi$. Hence, in what follows we shall consider the groups in Theorem \ref{partition} separately and investigate the connected components of its partition components. Since for $p$-groups, the connected components of corresponding proper power graphs are determined in Theorem \ref{pgroupcomponents}, in what follows we further assume that $G$ is not a $p$-group.
\begin{theorem}
Let $G$ be a Hughes-Thompson group and $p$ be a prime such that $H_p(G)\neq G$. Then the number of connected components $\p^*(G)$ is equal to $1+|G|/p$ if $H_p(G)$ is not a $q$-group and it is equal to $c(H_p(G))+|G|/p$, otherwise.
\end{theorem}
\begin{proof}
Let $p$ be a prime such that $H_p(G)\neq G$. Then $H_p(G)$ together with cyclic subgroups of order $p$ not in $H_p(G)$ form a partition $\Pi$ for $G$. Let $\Pi=\{H_p(G)\}\cup\{\gen{g_i}\}_{i=1}^n$. By \cite{drh-jgt:1959}, $H_p(G)$ is a nilpotent group, which is not a $p$-group. If $H_p(G)$ is not a $q$-group, then by Theorem \ref{nilpotentconnectivity}(2,3), $H_p(G)\setminus\{1\}$ is a connected component of $\p^*(G)$, otherwise by Theorem \ref{nilpotentconnectivity}(1), $H_p(G)\setminus\{1\}$ is a union of $c(H_p(G))$ connected components. On the other hand, $\gen{g_i}\setminus\{1\}$ are connected components of $\p^*(G)$. Now, the result follows from the fact that by \cite{drh-jgt:1959}, $[G:H_p(G)]=p$ and hence $n=|G|/p$.
\end{proof}

Le $G$ be a finite Frobenius group and $K,H$ denote the Frobenius kernel and a Frobenius complement of $G$, respectively. Then $K$ along with the conjugates of $H$ gives rise to a partition of $G$. It is known that $K$ is always nilpotent and hence the connectivity of $\p^*(K)$ is described in Theorem \ref{nilpotentconnectivity}. The structure of $H$ is also well understood and below we show that $\p^*(H)$ is indeed connected.
\begin{lemma}\label{fixedpointfreeautomorphismgroup}
Let $G$ be a group of fixed-point-free automorphisms of some finite group. Then $\p^*(G)$ is connected. If $G$ is solvable, then $\diam(\p^*(G))\leq6$. If $G$ is not solvable, then $\diam(\p^*(G))\leq12$ and the equality holds only if $G$ has a maximal subgroup $M$ of index $2$ such that $M=L\times SL(2,p)$ for some solvable group $L$ and prime $p$. In both cases, if $Z(G)\neq1$, then $\diam(\p^*(G))\leq4$.
\end{lemma}
\begin{proof}
First suppose that $G$ is solvable and $N$ is a minimal normal subgroup of $G$. By \cite[10.5.5]{djsr:1982}, $N=\gen{x}$ is cyclic of prime order $p$. Let $g\in G\setminus N$. Then $|g^n|=q$ is a prime for some integer $n$. If $g^n\in N$, then $g\sim x$, otherwise $|\gen{g^n,x}|=pq$ and by \cite[10.5.5]{djsr:1982}, $p\neq q$ and $\gen{g^n,x}$ is a cyclic group. Thus $g\simeq g^n\sim g^nx\sim x$. Hence $\p^*(G)$ is connected and $\diam(\p^*(G))\leq6$. 

Now, suppose that $G$ is not solvable. By \cite[10.5.6(ii)]{djsr:1982}, the Sylow $p$-subgroups of $G$ each of which is cyclic or a generalized quaternion $2$-group. Hence, by \cite{ms:1995}, $G$ has a subgroup $M$ such that $[G:M]\leq2$ and $M=L\times S$, where $L$ is solvable with all Sylow subgroups cyclic and $S\cong SL(2,p)$ for some odd prime $p$. Since $SL(2,p)$ has a unique involution, it follows, by Corollary \ref{uniqueinvolution}, that $\p^*(S)$ is connected. Similar to previous case, $\p^*(L)$ is connected too. 

We show that $\p^*(M)$ is connected. Let $(x,y)\in M\setminus\{(1,1)\}$ and $z$ be the unique involution of $S$. If $(x,y)$ is not a $2$-element, then $|(x,y)^n|=q$ is a an odd prime for some integer $n$ so that $(x,y)\simeq(x,y)^n\sim(x,yz)\sim(1,z)$. If $|(x,y)|=2^k$, then $|(x,y)^{2^{k-1}}|=2$ and $(x,y)\sim(x,y)^{2^{k-1}}$. Hence it is sufficient to show that all involutions of $M$ are connect to $(1,z)$ via a path. Let $(x,y)\in M$ be an involution different from $(1,z)$. Hence $|x|=2$ and $y=1$ or $z$. Let $y'\in S$ be an element of order $p$. Then
\[(x,y)\sim (x,yy') \sim (1,y')\sim (1,y'z)\sim (1,z),\]
which implies that $\p^*(M)$ is connected. If $[G:M]=1$, then $G=M$ and there is noting to prove. Thus we may assume that $[G:M]=2$. Let $g\in G\setminus M$. Then $g^2\in M$. If $g^2\neq 1$, then by the statements above $g$ is connected to $(1,z)$ via a path. Finally, suppose that $g^2=1$. Since the Sylow $2$-subgroups of $G$ are generalized quaternion $2$-groups, there exists a $2$-element $g'$ such that $g=g'^2\in M$, which is a contradiction. Therefore $\diam(\p^*(G))\leq12$ and $\diam(\p^*(G))\leq10$ if $G=M$.

Now, suppose that $Z(G)\neq1$ and $z\in Z(G)\setminus\{1\}$ be an element of prime order $q$. Let $x\in G$. If $q\nmid|x|$, then $x\sim xz\sim z$ and $d(x,z)=2$. Also, if $q||x|$, then $z\in\gen{x}$ and hence $d(x,z)=1$ for otherwise $\gen{x,z}$ has a subgroup isomorphic to $\Z_q\times\Z_q$, which is a contradiction. Therefore $\diam(\p^*(G))\leq4$. The proof is complete.
\end{proof}
\begin{theorem}\label{Frobenius}
Le $G$ be a Frobenius group with kernel $K$ and complement $H$. Then $\p^*(H)$ is connected and the number of connected components of $\p^*(G)$ is $|K|+1$ if $K$ is not a $p$-group and it is $|K|+c(K)$ if $K$ is a $p$-group.
\end{theorem}
\begin{proof}
The result follows by Theorem \ref{nilpotentconnectivity} and Lemma \ref{fixedpointfreeautomorphismgroup}.
\end{proof}
\begin{theorem}
If $G=PGL(2,p^n)$ ($p$ odd), then the number of connected components of $\p^*(G)$ is equal to $(p^{2n+1}-1)/(p-1)$.
\end{theorem}
\begin{proof}
By \cite[II, Satz 8.5]{bh:1967}, $G$ has three subgroups $\calH$, $\calK$ and $\calL$ whose conjugates gives rise to a partition of $G$. Moreover, $\calH$ is an elementary Abelian $p$-group of order $p^n$, $\calK$ is a cyclic group of order $p^n-1$ and $\calL$ is a cyclic group of order $p^n+1$. Also, $[G:N_G(\calH)]=p^n+1$, $[G:N_G(\calK)]=p^n(p^n+1)/2$ and $[G:N_G(\calL)]=p^n(p^n-1)/2$. Now, the result follows by Theorem \ref{nilpotentconnectivity}.
\end{proof}
\begin{theorem}
If $G=PSL(2,p^n)$, then the number of connected components of $\p^*(G)$ is equal to $(p^{2n+1}-1)/(p-1)$.
\end{theorem}
\begin{proof}
Similar to the previous theorem, $G$ has three subgroups $\calH$, $\calK$ and $\calL$ whose conjugates gives rise to a partition of $G$. Moreover, $\calH$ is an elementary Abelian $p$-group of order $p^n$, $\calK$ is a cyclic group of order $(p^n-1)/d$ and $\calL$ is a cyclic group of order $(p^n+1)/d$, where $d=\gcd(p-1,2)$. Also, $[G:N_G(\calH)]=p^n+1$, $[G:N_G(\calK)]=p^n(p^n+1)/2$ and $[G:N_G(\calL)]=p^n(p^n-1)/2$. Now, the result follows by Theorem \ref{nilpotentconnectivity}.
\end{proof}
\begin{theorem}
If $G=Sz(q)$, then the number of connected components of $\p^*(G)$ is equal to $\frac{1}{2}q^3(q+1)^2-q^2+q-1$, where $q=2^{2n+1}$.
\end{theorem}
\begin{proof}
By \cite[Theorem 3.10]{bh-nb:1982}, $G$ has four subgroups $\calH$, $\calK$, $\calU_1$ and $\calU_2$ whose conjugates gives rise to a partition of $G$. Moreover, $\calH$ is a $2$-group of order $q^2$, $\calK$ is cyclic of order $q-1$, $\calU_1$ is cyclic of order $q+2r+1$ and $\calU_2$ is cyclic of order $q-2r+1$, where $r=2^n$. Also, $[G:N_G(\calH)]=q^2+1$, $[G:N_G(\calK)]=q^2(q^2+1)/2$, $[G:N_G(\calU_1)]=q^2(q^4-1)/4(q+2r+1)$ and $[G:N_G(\calU_2)]=q^2(q^4-1)/4(q-2r+1)$. On the other hand, by Theorem \ref{nilpotentconnectivity}(1), the number of connected components of $\p^*(\calH)$ is equal to the number of involutions of $\calH$. Since $\calH=\{S(a,b):a,b\in GF(q)\}$ and for all $a,a',b,b'\in GF(q)$, the product in $\calH$ is defined as
\[S(a,b)S(a',b')=S(a+a',b+b'+a^\pi a'),\]
where $\pi\in\Aut(GF(q))$ is the unique automorphism of order $2$, it follows that $S(a,b)\in\calH$ is an involution if and only if $a=0$ and $b\neq0$. Hence $\calH$ has exactly $q-1$ involutions. Now, the result follows by Theorem \ref{nilpotentconnectivity}.
\end{proof}

We conclude this section by characterizing all finite groups whose proper power graphs have prescribed connected components.
\begin{theorem}
Let $G$ be a finite group such that the connected components of $\p^*(G)$ each of which adopts with the subgraph induced by nontrivial elements of some subgroup of $G$. Then either $\p^*(G)$ is connected and $G$ has only the trivial partition, or $\p^*(G)$ is disconnected and $G$ is isomorphic to one of the following groups:
\begin{itemize}
\item[(1)]a non-cyclic $p$-group with $H_p(G)\neq G$ such that $H_p(G)$ is cyclic or it is a generalized quaternion $2$-group,
\item[(2)]a group of Hughes-Thompson type such that if $H_p(G)\neq G$, then either $H_p(G)$ is not a $q$-group or $H_p(G)$ is a $q$-group and $H_q(H_p(G))$ is cyclic or it is a generalized quaternion $2$-group,
\item[(3)]a Frobenius group with kernel $K$ such that either $K$ is not a $p$-group or $K$ is a $p$-group and $H_p(K)$ is cyclic or it is a generalized quaternion $2$-group,
\item[(4)]$PGL(2,p^n)$ ($p$ odd),
\item[(5)]$PSL(2,p^n)$,
\end{itemize}
where $p$ is a prime number.
\end{theorem}
\begin{proof}
If $\p^*(G)$ is connected, then clearly $G$ has only the trivial partition and we are done. Now suppose that $\p^*(G)$ is disconnected. Then by hypothesis, $G$ has a non-trivial partition. Now, we consider the groups in Theorem \ref{partition} one by one.

(1) If $G$ is a $p$-group, then since $H_p(G)=H_p(H_p(G))$, the subgraph of $\p^*(G)$ induced by $H_p(G)\setminus\{1\}$ is connected and by Theorem \ref{nilpotentconnectivity}(1), $H_p(G)$ is a cyclic $p$-group or it is a generalized quaternion $2$-group.

(2) If $G$ is a group of Hughes-Thompson type and $H_p(G)\neq G$ for some prime $p$, then either $H_p(G)$ is not a $q$-group, hence the subgraph of $\p^*(G)$ induced by $H_p(G)\setminus\{1\}$ is connected, or $H_p(G)$ is a $q$-group and by Theorem \ref{nilpotentconnectivity}(1) and part (1), $H_q(H_p(G))$ is a cyclic $q$-group or it is a generalized quaternion $2$-group.

(3) If $G$ is a Frobenius group with kernel $K$, then since $K$ is nilpotent the result follows the same as in part (2).

(4) and (5) If $G=PGL(2,p^n)$ or $PSL(2,p^n)$, then there is nothing to prove.

(6) If $G=Sz(q)$, then since $\p^*(\mathcal{H})$ has $q-1>1$ involutions, by Theorem \ref{nilpotentconnectivity}(1), $\p^*(\mathcal{H})$ is disconnected, hence $\mathcal{H}$ should have a non-trivial partition, which contradicts the fact that by \cite{ohk:1961}, $\mathcal{H}$ has only the trivial partition. The proof is complete.
\end{proof}
\section{Symmetric and alternating groups}
This section is devoted to the connectivity of symmetric and alternating groups. Before stating the main theorems, we obtain distances between special elements of each group. In the sequel, $p$ and $q$ stand for prime numbers. We begin with studying symmetric groups.
\begin{lemma}\label{transposition}
For $n\geq7$, there exists a path of length four in $\p^*(S_n)$ between any two transpositions of $S_n$.
\end{lemma}
\begin{proof}
Suppose $(i\ j)$ and $(k\ l)$ are two transpositions in $S_n$. Let $u,v,w\in\{1,2,\ldots ,n\}\setminus\{i,j,k,l\}$ be three distinct letters. Then, by Lemma \ref{coprimeelements}
\[(i\ j)\sim(i\ j)(u\ v\ w)\sim(u\ v\ w)\sim(k\ l)(u\ v\ w)\sim(k\ l),\]
as required.
\end{proof}
\begin{theorem}\label{symmetric}
Let $G=S_n$ be a symmetric group ($n\geq2$). Then
\begin{itemize}
\item[(1)]if $n\geq9$ and neither $n$ nor $n-1$ is a prime, then $\p^*(G)$ is connected and $\diam(\p^*(G))\leq26$.
\item[(2)]if $n=p\geq11$, then the number of connected components of $\p^*(G)$ is equal to $(p-2)!+1$.
\item[(3)]if $n=p+1\geq12$, then the number of connected components of $\p^*(G)$ is equal to $(p+1)(p-2)!+1$.
\item[(4)]if $n=2,3,4,5,6,7,8$ , then the number of connected components of $\p^*(G)$ is equal to $1$, $4$, $13$, $31$, $83$, $541$ and $961$, respectively.
\end{itemize}
\end{theorem}
\begin{proof}
(1) Let $\pi\in G$ be an element of prime order $p$. Then $\pi=\pi_1\ldots\pi_k$, where $\pi_i$ are disjoint $p$-cycles. We show that $\pi$ is connected to a transposition via a path in $\p^*(S_n)$. First suppose that $p=2$. If $k=1$, then there is nothing to prove. If $k=2,3$, then $\pi=(i_1\ i_2)(j_1\ j_2)(k_1\ k_2)^\epsilon$, where $\epsilon=k-2$. Let $l_1,l_2,l_3\in\{1,\ldots,n\}\setminus\{i_1,i_2,j_1,j_2,k_1,k_2\}$ be distinct letters. Also, let $\sigma=(i_1 \ i_2)(j_1\ j_2)(k_1\ k_2)^\epsilon(l_1\ l_2\ l_3)$ and $\tau=(i_1\ i_2)(l_1\ l_2\ l_3)$. Hence $\pi\sim\sigma\sim\sigma^2\sim\tau\sim(i_1\  i_2)$. Now suppose that $k\geq4$ and $\pi_1=(i_1\ i_2)$, $\pi_2=(j_1\ j_2)$, $\pi_3=(k_1\ k_2)$. Let $\sigma=(i_1\ j_1\ k_1\ i_2\ j_2\ k_2)\pi_4\ldots\pi_k$ and $\tau=(i_1\ j_1\ k_1\ i_2\ j_2\ k_2)$. Then $\pi\sim\sigma\sim\sigma^2=\tau^2 \sim\tau\sim\tau^3$ and the same as before $\tau^3$ is connected to a transposition via a path in $\p^*(S_n)$. Finally, suppose that $p$ is an odd prime. If $k=1$, then $\pi=(i_1\ i_2\ \ldots\ i_p)$. Since $n\geq p+2$, we have two letters $u,v$ different from $i_1,\ldots,i_p$. Let $\sigma=(i_1\ i_2\ \ldots\ i_q)(u v)$. Then $\pi\sim\pi^2=\sigma^2\sim\sigma\sim\sigma^p=(u,v)$. Also, if $k>1$, $\pi_1=(i_1\ i_2\ \ldots\ i_p)$ and $\pi_2=(j_1\ j_2\ \ldots\ j_p)$, then by putting $\sigma=(i_1\ j_1\ i_2\ j_2\ \ldots\ i_p\ j_p)\pi_3^{(p+1)/2}\ldots \pi_k^{(p+1)/2}$, we have $\pi\sim\sigma\sim\sigma^p$. But $|\sigma^p|=2$ and by the previous cases $\sigma^p$ is connected to a transposition via a path. Now, by using Lemma \ref{transposition} in conjunction with the fact that each non-trivial permutation is adjacent to a permutation of prime order, $\p^*(G)$ is connected.

(2) Let $\pi\in G$. If $|\pi|=p$, then clearly $\gen{\pi}\setminus\{1\}$ is a connected component of $\p^*(G)$. On the other hand, if $|\pi|\neq p$ and $q$ is a prime divisor of $|\pi|$, then $q\neq p$ and $|\pi^{|\pi|/q}|=q$. Similar to (1), it can be shown that $\pi^{|\pi|/q}$ and hence $\pi$ is connected to a transposition via a path. Hence, by Lemma \ref{transposition} in conjunction with the fact that the number of distinct cyclic subgroups of $S_p$ of order $p$ equals $p!/p(p-1)$, it follows that $\p^*(G)$ has $(p-2)!+1$ connected components.

(3) The proof is similar to (2).

(4) If $n=2,3,4,5,6$, then the result follows easily. If $n=7$, then the sets $\gen{\pi}\setminus\{1\}$, where $\pi\in S_n$ is an cycle of order $6$ or $7$ together with the set of remaining elements each of which is a connected component of $\p^*(S_n)$. Hence $\p^*(S_n)$ has exactly $7!/7\cdot6+7!/6\cdot2+1=541$ connected components. If $n=8$, then the sets $\gen{\pi}\setminus\{1\}$, where $\pi\in S_n$ is an element of order $7$ together with the set of remaining elements each of which is a connected component of $\p^*(S_n)$. Hence $\p^*(S_n)$ has exactly $8!/7\cdot6+1=961$ connected components. The proof is complete.
\end{proof} 

In the sequel, we discuss on the connectivity of $\p^*(A_n)$. For each $\pi\in A_n$, the \textit{support} of $\pi$ is defined by $\Supp(\pi)=\{i:\pi(i)\neq i\}$. A permutation that is a product of $n$ disjoint $m$-cycles is called an $(m,n)$-element. Also the radical of a subgroup $H$ of a group $G$, denoted by $\sqrt{H}$, is th set of all elements $g\in G$ such that $1\neq g^n\in H$ for some integer $n$.
\begin{lemma}\label{threecycle}
For $n\geq 10$, there is a path of length at most four in $\p^*(A_n)$ between every two three cycles. 
\end{lemma}
\begin{proof}
Let $(i_1\ i_2\ i_3)$ and $(j_1\ j_2\ j_3)$ be two $3$-cycles in $A_n$. Since $n\geq10$, there exists four distinct letters $l_1,l_2,l_3,l_4$ different from $i_1,i_2,i_3,j_1,j_2,j_3$. Then
\[(i_1\ i_2\ i_3)\sim(i_1\ i_2\ i_3)(l_1\ l_2)(l_3\ l_4)\sim(l_1\ l_2)(l_3\ l_4)\sim(l_1\ l_2)(l_3\ l_4)(j_1\ j_2\ j_3)\sim(j_1\ j_2\ j_3),\]
as required.
\end{proof}
\begin{lemma}\label{order2}
For $n\geq 10$, every involution in $A_n$ is connected to a $3$-cycle via a path of length at most six.
\end{lemma}
\begin{proof}
Let $\pi\in A_n$ be an involution. If $|\Supp(\pi)|\leq n-3$, then there exists three distinct letters $i_1,i_2,i_3\in\{1,\ldots,n\}\setminus\Supp(\pi)$. Thus $\pi\sim\pi(i_1\ i_2\ i_3)\sim(i_1\ i_2\ i_3)$. Now, suppose that $|\Supp(\pi)|\geq n-2$. Hence $\pi=(i_1\ i_2)(j_1\ j_2)(k_1\ k_2)(l_1\ l_2)\sigma$ for some involution $\sigma$ such that $i_1,i_2,j_1,j_2,k_1,k_2,l_1,l_2\notin\Supp(\sigma)$. Now, define $\pi'=(i_1\ j_1\ k_1\ i_2\ j_2\ k_2)(l_1\ l_2)\sigma$. We have
\[\pi\sim\pi'\sim\pi'^2\sim\pi'^2(l_1\ l_2)(h_1\ h_2)\sim(l_1\ l_2)(h_1\ h_2)\sim(l_1\ l_2)(h_1\ h_2)(i_1\ i_2\ j_1)\sim(i_1\ i_2\ j_1),\]
where $h_1,h_2\in\{1,\ldots,n\}\setminus\{i_1,i_2,j_1,j_2,k_1,k_2,l_1,l_2\}$. The proof is complete.
\end{proof}
\begin{lemma} \label{order3}
For $n\geq11$, every element of order $3$ in $A_n$ is connected to a $3$-cycle via a path of length at most eight.
\end{lemma}
\begin{proof}
Let $\pi\in A_n$ be an element of order $3$. If $|\Supp(\pi)|\geq12$, then 
\[\pi=(i_1\ i_2\ i_3)(j_1\ j_2\ j_3)(k_1\ k_2\ k_3)(l_1\ l_2\ l_3)\sigma\]
for some element $\sigma$ of order $3$ such that $i_t,j_t,k_t,l_t\notin\Supp(\sigma)$, for $t=1,2,3$. Then 
\[\pi\sim\tau_1\tau_2\sigma^{-1}\sim(i_1\ j_2)(i_2\ j_3)(i_3\ j_1)(k_1\ l_2)(k_2\ l_3)(k_3\ l_1),\]
where $\tau_1=(i_1\ j_1\ i_2\ j_2\ i_3\ j_3)$ and $\tau_2=(k_1\ l_1\ k_2\ l_2\ k_3\ l_3)$. Hence, by Lemma \ref{order2}, $\pi$ is connected to a $3$-cycle via a path. Now suppose that $|\Supp(\pi)|\leq9$. If $|\Supp(\pi)|=3$, then $\pi$ is a $3$-cycle and there is nothing to prove. Thus we suppose that $|\Supp(\pi)|\geq6$. Then $\pi=(i_1\ i_2\ i_3)(j_1\ j_2\ j_3)\sigma$, were $\sigma=1$ or a $3$-cycle. Since $|\Supp(\pi)|\leq n-2$, we can chose two distinct letters $k_1,k_2\in\{1,\ldots,n\}\setminus\Supp(\pi)$. Then
\[\pi\sim(i_1\ j_1\ i_2\ j_2\ i_3\ j_3)\sigma^{-1}(k_1\ k_2)\sim(i_1\ j_2)(i_2\ j_3)(i_3\ j_1)(k_1\ k_2)\]
and by Lemma \ref{order2}, $\pi$ is connected to a $3$-cycle via a path, as required.
\end{proof}
\begin{lemma}\label{orderprime}
If $\pi\in A_n$ ($n\geq11$) is an element of prime order $p\neq2,3$ and either $|\Supp(\pi)|\leq n-3$, or $|\Supp(\pi)|\geq3p,n-2$, then $\pi$ is connected to a $3$-cycle via a path of length at most six.
\end{lemma}
\begin{proof}
If $|\Supp(\pi)|\leq n-3$, then $\pi\sim\pi(i_1\ i_2\ i_3)\sim(i_1\ i_2\ i_3)$, where $i_1,i_2,i_3\in\{1,\ldots,n\}\setminus \Supp(\pi)$. Now, if $|\Supp(\pi)|\geq n-2$ and $|\Supp(\pi)|\geq3p$, then 
\[\pi=(i_1\ i_2\ \cdots\ i_p)(j_1\ j_2\ \cdots\ j_p)(k_1\ k_2\ \cdots\ k_p)\sigma\]
for some element $\sigma$ of order $p$ such that $i_t,j_t,k_t\notin\Supp(\sigma)$, for $t=1,\ldots,p$. Let $\pi'=(i_1\ j_1\ k_1\ i_2\ j_2\ k_2\ \cdots\ i_p\ j_p\ k_p)\sigma^s$, in which $\sigma^{3s}=\sigma$. Then $\pi\sim\pi'\sim\pi'^p$ . But $|\pi'^p|=3$ and by Lemma \ref{order3}, $\pi'^p$ is connected to a $3$-cycle via a path of length at most four, as required.
\end{proof}
\begin{theorem}\label{alternating}
Let $G=A_n$ be an alternating group ($n\geq3$). Then
\begin{itemize}
\item[(1)]if $n,n-1,n-2,n/2,(n-1)/2,(n-2)/2$ are not primes, then $\p^*(G)$ is connected and $\diam(\p^*(G))\leq22$.
\item[(2)]if $n=p\geq11$, then the number of connected components of $\p^*(G)$ is equal to $(p-2)!+p(p-1)(p-4)!/2+1$ if $p-2$ is prime, it is equal to $(p-2)!+4p(p-2)(p-4)!/(p-1)+1$ if $(p-1)/2$ is prime and it is equal to $(p-2)!+1$ if neither $p-2$ nor $(p-1)/2$ is a prime.
\item[(3)]if $n=p+1\geq12$, then the number of connected components of $\p^*(G)$ is equal to $(p+1)(p-2)!+4p(p-2)!/(p+1)+1$ if $(p+1)/2$ is prime, it is equal to $(p+1)(p-2)!+4p(p+1)(p-2)(p-4)!/(p-1)+1$ if $(p-1)/2$ is prime and it is equal to $(p+1)(p-2)!+1$ if neither $(p+1)/2$ nor $(p-1)/2$ is a prime.
\item[(4)]if $n=p+2\geq13$, then the number of connected components of $\p^*(G)$ is equal to $p!+(p+2)(p+1)(p-2)!/2+(p+2)(p+1)(p-2)!/2+1$ if $p+2$ and $(p+1)/2$ are primes, it is equal to $p!+(p+2)(p+1)(p-2)!/2+1$ if $p+2$ is prime but $(p+1)/2$ is not prime, it is equal to $(p+2)(p+1)(p-2)!/2+4p(p+2)(p-2)!/(p+1)+1$ if $(p+1)/2$ is prime but $p+2$ is not prime and it is equal to $p!+1$ if neither $p+2$ nor $(p+1)/2$ is prime.
\item[(5)]if $n=2p\geq14$, then the number connected components of $\p^*(G)$ is equal to $2p(2p-3)!+(2p-1)!/p(p-1)+1$ if $2p-1$ is prime and it is equal to $(2p-1)!/p(p-1)+1$ when $2p-1$ is not prime.
\item[(6)]if $n=2p+1\geq11$, then the number of connected components of $\p^*(G)$ is equal to $(2p+1)(2p-1)!/p(p-1)+(2p-1)!+1$ if $2p+1$ is prime, it is equal to $(2p+1)(2p-1)!/p(p-1)+p(2p+1)(2p-3)!+1$ if $2p-1$ is prime and it is equal to $(2p+1)(2p-1)!/p(p-1)+1$ otherwise.
\item[(7)]if $n=2p+2\geq12$, then $\p^*(G)$ is connected when $2p+1$ is not prime and it is disconnected with $2(p+1)(2p-1)!+1$ connected components if $2p+1$ is prime.
\item[(8)]if $n=3,4,5,6,7,8,9,10$, then the number of connected components of $\p^*(G)$ is equal to $1$, $7$, $31$, $121$, $421$, $842$, $5442$ and $29345$, respectively.
\end{itemize}
\end{theorem}
\begin{proof}
By Lemmas \ref{threecycle}, \ref{order2}, \ref{order3} and \ref{orderprime}, an element $\pi\in A_n$ of prime order $q$ where $q=2,3$ or $q>3$ and either $|\Supp(\pi)|\leq n-3$, or $|\Supp(\pi)|\geq n-2$ and $|\Supp(\pi)|\geq3q$ is connected to a $3$-cycle when $n\geq11$. Thus in cases (1) to (7) we discuss just the cases where $|\Supp(\pi)|\geq n-2$ and $|\Supp(\pi)|=q$ or $2q$. Actually $|\Supp(\pi)|\leq n\leq|\Supp(\pi)|+2$.

(1) It is easy to see that $n\geq52$. Since $n,n-1,n-2,n/2,(n-1)/2,(n-2)/2$ are not primes and $|\Supp(\pi)|\leq n\leq|\Supp(\pi)|+2$, it follows that $|\Supp(\pi)|\neq q$ and $2q$. Hence, by invoking Lemma \ref{threecycle} in conjunction with the fact that each non-trivial even permutation is adjacent to an even permutation of prime order, $\p^*(G)$ is connected.

(2) If $|\Supp(\pi)|=2q$, then $q=(p-1)/2$. Clearly, $\gen{\pi}\setminus\{1\}$ is a connected component of $\p^*(G)$ and the number of such components is equal to $4p(p-2)(p-4)!/(p-1)$. Also, if $|\Supp(\pi)|=q$, then either $q=p$ or $p-2$. Clearly, $\gen{\pi}\setminus\{1\}$ is a connected component of $\p^*(G)$ and the number of such components is equal to $(p-2)!$ if $q=p$ and it is equal to $p(p-1)(p-4)!/2$ if $q=p-2$. Since $p-2$ and $(p-1)/2$ are not prime simultaneously, the results follow.

(3) If $|\Supp(\pi)|=2q$, then $q=(p+1)/2$ or $(p-1)/2$. Clearly, $\gen{\pi}\setminus\{1\}$ is a connected component of $\p^*(G)$ and the number of such components is equal to $4p(p-2)!/(p+1)$ if $q=(p+1)/2$ and it is equal to $2p(p+1)(p-2)(p-4)!/(p-1)$ if $q=(p-1)/2$. Also, if $|\Supp(\pi)|=q$, then $q=p$. Clearly, $\gen{\pi}\setminus\{1\}$ is a connected component of $\p^*(G)$ and the number of such components is equal to $(p+1)(p-2)!$. Since $(p+1)/2$ and $(p-1)/2$ are not prime simultaneously, the results follow.

(4) If $|\Supp(\pi)|=2q$, then $q=(p+1)/2$. Clearly, $\gen{\pi}\setminus\{1\}$ is a connected component of $\p^*(G)$ and the number of such components is equal to $4p(p+2)(p-2)!/(p+1)$. Also, if $|\Supp(\pi)|=q$, then either $q=p$ or $q=p+2$. Clearly, $\gen{\pi}\setminus\{1\}$ is a connected component of $\p^*(G)$ and the number of such components is equal to $(p+2)(p+1)(p-2)!/2$ when $q=p$ and it is equal to $p!$ when $q=p+2$. Now, the results follow by considering the cases where $p+2$ and $(p+1)/2$ are primes or not.

(5) If $|\Supp(\pi)|=2q$, then $q=p$. Clearly, $\gen{\pi}\setminus\{1\}$ is a connected component of $\p^*(G)$ and the number of such components is equal to $(2p-1)!/p(p-1)$. If $|\Supp(\pi)|=q$, then $q=2p-1$. Clearly, $\gen{\pi}\setminus\{1\}$ is a connected component of $\p^*(G)$ and the number of such components is equal to $2p(2p-3)!$, from which the results follow.

(6) If $|\Supp(\pi)|=2q$, then $q=p$. It is easy to see that $\gen{\pi}\setminus\{1\}$ is a connected component of $\p^*(G)$ and the number of such components is equal to $(2p+1)(2p-1)!/p(p-1)$. If $|\Supp(\pi)|=q$, then either $q=2p+1$ or $q=2p-1$.  If $q=2p+1$, then $\gen{\pi}\setminus\{1\}$ is a connected component of $\p^*(G)$ and the number of such components is equal to $(2p-1)!$. Also, if $q=2p-1$, then again $\gen{\pi}\setminus\{1\}$ is a connected component of $\p^*(G)$ and the number of such components is equal to $p(2p+1)(2p-3)!$. Hence the results follow.

(7) If $|\Supp(\pi)|=2q$, then $q=p$ and so $\pi=(i_1\ i_2\ \cdots\ i_p)(j_1\ j_2\ \cdots\ j_p)$. Let $\pi'=(i_1\ j_1\ i_2\ j_2\ \cdots\ i_p\ j_p)(k_1\ k_2)$, where $k_1,k_2\in\{1,\ldots,n\}\setminus\{i_1,\ldots,i_p,j_1,\ldots,j_p\}$. Then $\pi'\in G$ and $|\pi'^p|=2$.  Also $\pi\sim\pi'\sim\pi'^p$. By Lemma \ref{order2}, $\pi'^p$ and hence $\pi$ is connected to a $3$-cycle via a path. If $|\Supp(\pi)|=q$, then $q=2p+1$. Clearly, $\gen{\pi}\setminus\{1\}$ is a connected component of $\p^*(G)$ and the number of such components is equal to $2(p+1)(2p-1)!$, from which the results follow.

(8) If $n=2,3,4,5,6,7$, then the results follow by simple computations. Let $n=8$ and for $\{i_1,i_2,j_1,j_2,k_1,k_2,l_1,l_2\}=\{1,\ldots,8\}$, let $\sim'$ denote the following path
\[\begin{array}{ccc}
(i_1\ i_2)(j_1\ j_2)(k_1\ k_2)(l_1\ l_2)&=&(i_1\ i_2)(j_2\ j_1)(k_1\ k_2)(l_1\ l_2)\\
&&\wr\\
&&(i_1\ j_2\ k_1\ i_2\ j_1\ k_2)(l_1\ l_2)\\
&&\wr\\
(i_1\ k_1\ j_1)(i_2\ k_2\ j_2)&=&(i_1\ k_1\ j_1)(j_2\ i_2\ k_2)\\
\wr&&\\
(i_1\ i_2\ k_1\ k_2\ j_1\ j_2)(l_1,\ l_2)\\
\wr&&\\
(i_1\ k_2)(i_2\ j_1)(k_1\ j_2)(l_1\ l_2)&=&(i_1\ k_2)(j_1\ i_2)(k_1\ j_2)(l_1\ l_2)
\end{array}\]
Then
\[\begin{array}{ccc}
(i_1\ i_2)(j_1\ j_2)(k_1\ k_2)(l_1\ l_2)&&\\
\wr'&&\\
(i_1\ k_2)(j_1\ i_2)(k_1\ j_2)(l_1,\ l_2)&=&(i_1\ k_2)(k_1\ j_2)(i_2\ j_1)(l_1\ l_2)\\
&&\wr'\\
(i_1\ j_1)(i_2\ j_2)(k_1\ k_2)(l_1\ l_2)&=&(i_1\ j_1)(k_1\ k_2)(i_2\ j_2)(l_1\ l_2)
\end{array}\]
so that each involution $\pi\in G$ with support of size $8$ is connected to all involutions obtained from $\pi$ by changing  two arbitrary letters. Hence, all involutions with supports of size $8$ belong to the same connected component of $\p^*(G)$. On the other hand, if $\{i_1,\ldots,i_5\}\subseteq\{1,\ldots,n\}$, $\{j_1,j_2,j_3\}\subseteq\{1,\ldots,n\}\setminus\{i_1,\ldots,i_5\}$ and $\{k_1,\ldots,k_4\}\subseteq\{1,\ldots,n\}\setminus\{j_1,j_2,j_3\}$, then
\[(i_1\ i_2\ i_3\ i_4\ i_5)\sim(i_1\ i_2\ i_3\ i_4\ i_5)(j_1\ j_2\ j_3)\sim(j_1\ j_2\ j_3)\]
and
\[(j_1\ j_2\ j_3)\sim(j_1\ j_2\ j_3)(k_1\ k_2)(k_3\ k_4)\sim(k_1\ k_2)(k_3\ k_4),\]
from which it follows that all $5$-cycles, $3$-cycles and involutions with supports of size $4$ belong to the same connected component of $\p^*(G)$. Now the result follows, since for each $7$-cycle $\pi\in G$, $\gen{\pi}\setminus\{1\}$ is a connected component of $\p^*(G)$ and the number of such components is equal to $8!/8\cdot6=840$.

If $n=9$, then the same as for $A_8$, all $5$-cycles, $3$-cycles and involutions with supports of size $4$ belong to the same connected component. Also, all elements of order $3$ with supports of size $6$ and involutions with supports of size $8$ belong to the same connected component. On the other hand, $\sqrt{\gen{\pi}}$, where $\pi$ is a $7$-cycle or a $(3,3)$-element is a connected component of $\p^*(G)$ and the number of such components is equal to $4320$ and $1120$, respectively. Hence the result follows.

If $n=10$, then by Lemmas \ref{order2} and \ref{threecycle}, all involutions and $3$-cycles belong to the same connected component $\mathcal{C}$ of $\p^*(G)$. If $\{i_1,\ldots,i_7,j_1,j_2,j_3\}=\{1,\ldots,10\}$ and $\{k_1,k_2,l_1,l_2\}=\{i_4,i_5,i_6,i_7\}$, then 
\begin{align*}
(i_1\ i_2\ i_3\ i_4\ i_5\ i_6\ i_7)&\sim(i_1\ i_2\ i_3\ i_4\ i_5\ i_6\ i_7)(j_1\ j_2\ j_3)\sim(j_1\ j_2\ j_3),\\
(i_1\ i_2\ i_3\ i_4\ i_5)&\sim(i_1\ i_2\ i_3\ i_4\ i_5)(j_1\ j_2\ j_3)\sim(j_1\ j_2\ j_3),\\
(i_1\ i_2\ i_3)(j_1\ j_2\ j_3)&\sim(i_1\ i_2\ i_3)(j_1\ j_2\ j_3)(k_1\ k_2)(l_1\ l_2)\sim(k_1\ k_2)(l_1\ l_2),
\end{align*}
which implies that all $7$-cycles, $5$-cycles and $(3,2)$-elements belong to $\mathcal{C}$.  On the other hand, $\sqrt{\gen{\pi}}$, where $\pi$ is a $(5,2)$-element or $(3,3)$-element is a connected component of $\p^*(G)$ and the number of such components is equal to $18144$ and $11200$, respectively. The proof is complete.
\end{proof}

The above results suggest us to pose the following conjecture.
\begin{conjecture}
There exists a constant $c$ such that for all finite groups $G$ with connected proper power graph, $\diam(\p^*(G))\leq c$.
\end{conjecture}

\end{document}